\documentclass{amsart}
\usepackage {amsmath, amscd}
\usepackage {amssymb}
\usepackage{hyperref,cite}
\usepackage{color}

\numberwithin{equation}{section}

\def\<{\langle}
\def\>{\rangle}

\def\EE{{\mathcal E}}

\def\II{{\mathcal I}}

\def\LL{{\mathcal L}}

\def\PP{{\mathcal P}}

\def\SS{{\mathcal S}}
\def\TT{{\mathcal T}}

\def\bbZ{\mathbb{Z}}
\def\bbN{\mathbb{N}}

\newcommand{\rank}{\mathop{\rm rank}}

\newcommand{\ran}{\mathop{\rm ran}}

\newcommand{\supp}{\mathop{\rm supp}}

\newtheorem{lemma}{Lemma}[section]
\newtheorem{theorem}[lemma]{Theorem}%[section]
\newtheorem{corollary}[lemma]{Corollary}

\theoremstyle{definition}
\newtheorem{remark}[lemma]{Remark}
\newtheorem{definition}[lemma]{Definition}

\newcommand{\Implies}{$\Rightarrow$}

\newcommand{\sime}{\stackrel{e}{\sim}}

\title[Schur coupling of operators on a  Hilbert space]{Schur coupling and related equivalence relations for operators on a  Hilbert space}
\author{Dan Timotin}

\address{Institute of Mathematics ``Simion Stoilow'' of the Romanian Academy, P.O. Box 1-764, Bucharest 
014700, Romania}
\email{Dan.Timotin@imar.ro}

\keywords{Equivalence after extension, Schur coupling, matricial coupling}

\subjclass{47A05, 47A64, 15A99}

\begin{document} 

\begin{abstract}
For operators on Hilbert spaces of any dimension, we show that equivalence after extension coincides with equivalence after one-sided extension, thus obtaining a  proof of their coincidence with Schur coupling. We also provide a concrete description of this equivalence relation in several  cases, in particular for compact operators.
\end{abstract}

\maketitle

\section{Introduction}

Motivated especially by applications to system theory and integral equations, several equivalence relations have been introduced, starting with~\cite{BGK}, for matrices and operators on Banach spaces: matricial coupling, equivalence after extension, Schur coupling, equivalence after one-sided extension (see  Section~\ref{se:prelim} for precise definitions  of the last three, with which we will be concerned). The coincidence of the first two was established in~\cite{BGK, BT1}. More discussion appears in~\cite{BT0}, where it is  proved that Schur coupling implies equivalence after extension. Equivalence after one-sided extension is introduced in~\cite{BGKR}; it is shown that it implies Schur coupling. In the case of separable Hilbert spaces, it has been recently proved in~\cite{HR} that equivalence after extension coincides with Schur coupling.

We restrict ourselves in this paper to Hilbert spaces. Without any separability assumption, we  show in this case that equivalence after extension coincides with equivalence after one-sided extension, thus providing a proof of the coincidence of all the discussed relations. The proof is quite transparent and may be used to obtain  concrete criteria for the equivalence of two operators.  In the case of compact operators, this becomes a simple relation between their respective singular values.

The tools used  stem from an older paper of Fillmore and Williams on operator ranges~\cite{FW} (which in turn is based on ideas of K\"othe~\cite{K1, K2}). The characterization of equivalence of operators therein (called below strong equivalence in order to avoid confusions) may be refined to provide descriptions of the other equivalence relations that concern us. 

The method uses spectral projections of selfadjoint operators and is therefore confined to the Hilbert space setting. The coincidence of the discussed equivalence relations  remains open in the case of a general infinite dimensional Banach space. It seems plausible that its validity should depend on geometric properties of the Banach space.

The plan of the paper is the following. After dealing with the preliminaries, we devote a whole section to our basic technical tool, a combinatorial lemma that we prefer to discuss separately without reference to Hilbert spaces and operators. The next section contains some preparatory  discussion of strong equivalence and equivalence after extension.  The main result, the coincidence of Schur coupling with equivalence after one-sided or two-sided extension, appears as Theorem~\ref{th:main} in Section~\ref{se:eoe}. It allows in Section~\ref{se:ch schur} to characterize concretely this equivalence relation for several classes, in particular for compact operators (see Theorems~\ref{th:schur coupling for compacts} and~\ref{th:main description2}).

\section{Preliminaries}\label{se:prelim}

All our operators will act between   Hilbert spaces that are not supposed to be separable. We denote by $\LL(X,Y)$ the space of linear operators from $X$ to $Y$; $\LL(X):=\LL(X,X)$.  The notations $\ker T$ and $\ran T$ indicate the kernel, respectively the range of the operator $T$. Note that $\ker T$ is always closed, while $\ran T$ might not be. Whenever $A\in\LL(X)$ is a positive operator,  $E_A$ denotes its spectral projection and $\supp A$ its support---that is, $\supp A=E_A((0,\infty))X=\ker A^\perp$.

Suppose $T\in\LL(X)$, $S\in\LL(Y)$ are bounded linear operators between  Hilbert spaces. We will define several notions of equivalence, as follows.

\begin{itemize}
\item[(E)] $T,S$ are called \emph{strongly equivalent} (notation $T\sim S$) if there exist invertible operators $U\in\LL(X,Y)$, $V\in\LL(Y,X)$, such that $S=UTV$. 

\item
[(EE)] $T,S$ are called \emph{equivalent after extension} if $T\oplus I_Y \sim S\oplus I_X$; we write this also $T\sime S$.
 
%If $\dim X=k$, $\dim Y=m$, where $k,m\in\bbN\cup\{\infty\}$, we will also say that $T,S$ are \emph{$(k,m)$-equivalent after extension}.

\item
[(EOE)] $T,S$ are called \emph{equivalent after one-sided extension} if either $T\oplus I_H\sim S$,  or $T\sim S\oplus I_H$ for some Hilbert space $H$. If $\dim H=k$, we say that $T,S$ are called \emph{equivalent after  $k$-one-sided extension}.

\item
[(SC)] $T,S$ are called \emph{Schur coupled} if there exists operators $A,B,C,D$ with $A,D$ invertible, such that $T=A-BD^{-1}C$ and $S=D-CA^{-1}B$.

\end{itemize}

A few remarks are in order. What we call ``strongly equivalent'' is often (for instance in~\cite{FW}) called just ``equivalent''; we prefer our choice in order to avoid possible confusions. The definition of equivalence after extension usually does not specify the space on which the identity operators act---one assumes it is just some Banach space. However,  it actually intervenes (in~\cite{BGK, BGKR, HR})  in the way we have stated it above.

Several implications are known to be true for these equivalence relations. Obviously (E)\Implies (EOE)\Implies (EE). It is shown in~\cite{BGKR} (Theorem~2 and Proposition~5) that (EOE)\Implies (SC)\Implies (EE). All these implications have an algebraic nature and are valid for operators acting in general Banach spaces.
Finally, the recent paper ~\cite{HR} shows that, if the operators act in separable Hilbert spaces, then (EE)\Implies (SC) (and thus Schur coupling and equivalence after extension coincide). The proof uses a result of Feldman and Kadison~\cite{FK} on the closure of invertible operators and is less constructive.

Below we will  show  directly, without any separability assumption, that on Hilbert spaces (EE)\Implies(EOE),  providing thus also a proof of the coincidence of Schur coupling with these relations.

A few immediate remarks that will  often be used without comment are contained in the next lemma.

\begin{lemma}\label{le:immediate remarks equivalence}
\begin{itemize}
\item
[(i)] If $T\sim S$, $T'\sim S'$, then $T\oplus T'\sim S\oplus S'$.
\item
[(ii)] If $T\sime S$, $T'\sime S'$, then $T\oplus T'\sime S\oplus S'$.
\item[(iii)]
Any two invertible operators on spaces of the same dimension are strongly equivalent.
\item[(iv)]
If $T\sime S$ (in particular, if $T\sim S$), then the following equalities are satisfied:
\begin{equation}\label{eq:index}
\dim\ker T = \dim\ker S,\quad \dim\ker T^* = \dim\ker S^*.
\end{equation}
\end{itemize}
\end{lemma}

The \emph{kernel condition}~\eqref{eq:index} will  often appear in the sequel.

The next result concerning ranges of positive operators is well-known (and easy to prove).

\begin{lemma}\label{le:range of selfadjoint}
Suppose $A$ is a positive operator, and $0<\delta<1$. For $n\in\bbZ$ denote by $\EE_n$ the image of $E_A([\delta^{n+1},\delta^n))$ (so $\EE_n=\{0\}$ for $-n$ sufficiently large). If $x=\sum_{n\in\bbZ} x_n$, $x_n\in\EE_n$ is the orthogonal decomposition of a vector in $\supp A$, then $x$ is in the range of $A$ if and only if $\sum_{n=0}^\infty \delta^{-2n}\|x_n\|^2<\infty$.
\end{lemma}

\section{A combinatorial lemma}\label{se:combin lemma}

This section is devoted to  a combinatorial lemma which is our main technical tool. It is essentially inspired by the proof of~\cite[Theorem 3.3]{FW}, but we need a refinement of the result therein.

\begin{lemma}\label{le:combinatorics} Let 
$\TT, \SS$ be two  sets, $N\ge 1 $ an integer, $0<\delta<1\le M$,  and $\tau:\TT\to(0,M]$, $\sigma:\SS\to(0,M]$  two functions. 

\begin{enumerate}
\item

Suppose that for any integers $k\ge N$, $\ell\ge 1$,  the following inequalities are satisfied:
\begin{align}
\label{eq:combinatorial lemma1}
\#\tau^{-1}( [\delta^{k+\ell}, \delta^{k}))&\le \#\sigma^{-1}( [\delta^{k+\ell+1}, \delta^{k-1}))\\
\#\sigma^{-1}( [\delta^{k+\ell}, \delta^{k}))&\le \#\tau^{-1}( [\delta^{k+\ell+1}, \delta^{k-1}))\label{eq:combinatorial lemma2}
\end{align}

Then for some $\delta'>0$ one of the following is true:
\begin{itemize}
\item
[(i)]
There exists a bijection $\eta:\TT\to\SS$ such that 
\begin{equation}\label{eq:delta' inequalities}
\delta'\le \frac{\tau(t)}{\sigma(\eta(t))}\le \frac{1}{\delta'} 
\end{equation}
for all $t\in\TT$.

\item
[(ii)] For some  set $\II$, if we extend $\tau$ to $\TT\cup\II$ by putting $\tau(t)=1$ for all $t\in\II$, then there exists a bijection $\eta:\TT\cup\II\to\SS$ such that~\eqref{eq:delta' inequalities} are true for all $t\in\TT\cup\II$.

\item
[(iii)] For some set $\II$, if we extend $\sigma$ to $\SS\cup\II$ by putting $\sigma(s)=1$ for all $s\in\II$, then there exists a bijection $\eta:\TT\to\SS\cup\II$ such that~\eqref{eq:delta' inequalities} are true for all $t\in\TT$.
\end{itemize}

\item

In case inequalities~\eqref{eq:combinatorial lemma1} and~\eqref{eq:combinatorial lemma2} are true for all $k\in\bbZ$, then \emph{(i)}  above is true.

\end{enumerate}
\end{lemma}

\begin{proof}
Denote, for $j\in\bbZ$, $\TT_j=\tau^{-1}([\delta^{j+1}, \delta^j))$ and $\SS_j=\tau^{-1}([\delta^{j+1}, \delta^j))$. Relations~\eqref{eq:combinatorial lemma1} and~\eqref{eq:combinatorial lemma2} say that, for $k\ge N$, 
\[
\begin{split}
\#(\bigcup_{j=k}^{k+\ell-1}\TT_j)\le \#(\bigcup_{j=k-1}^{k+\ell}\SS_j),\\
\#(\bigcup_{j=k}^{k+\ell-1}\SS_j)\le \#(\bigcup_{j=k-1}^{k+\ell}\TT_j).
\end{split}
\]
It is easy to see that one may extend these inequalities to nonconsecutive values of~$j$; that is,
\begin{align}
\#(\bigcup_{i=1}^I \TT_{j_i}) \le \#(\bigcup_{i=1}^I (\SS_{j_i-1}\cup\SS_{j_i}\cup\SS_{j_i+1}),\label{eq:comb2-1}\\
\#(\bigcup_{i=1}^I \SS_{j_i}) \le \#(\bigcup_{j=1}^J (\TT_{j_i-1}\cup\TT_{j_i}\cup\SS_{j_i+1}).\label{eq:comb2-2},
\end{align}
as long as all $j_i\ge N$.

Denote  
\[
\begin{split}
\TT'&=\bigcup_{j=N}^\infty \TT_j,\\
\TT''&=\bigcup \{\TT_j: j\ge N,\  \SS_{j-1}\cup \SS_j\cup\SS_{j+1}\text{ finite}\},\\
\tilde\SS&=\bigcup \{\SS_j: \SS_j\text{ finite}\}.
\end{split}
\]
We will define a map $p:\TT'\to \PP(\SS)$ with the property that, if $t\in\TT_j$, then $p(t)$ is a \emph{finite} subset of $\SS_{j-1}\cup \SS_j\cup\SS_{j+1}$.  If $t\in\TT_j\subset\TT''$, we put simply $p(t)=\SS_{j-1}\cup \SS_j\cup\SS_{j+1}$. 

If $t\in\TT'\setminus \TT''$ we  use a different procedure.
We start by defining an injection $p':\TT'\setminus \TT''\to \SS\setminus \tilde \SS$, as follows.
If $\TT_j\subset \TT'\setminus \TT''$, then at least one of $\SS_{j-1}, \SS_{j}, \SS_{j+1}$, say $\SS_{q_j}$, is infinite, and 
it satisfies $\#\SS_{q_j}\ge \#\TT_{j}$ (by~\eqref{eq:combinatorial lemma1} with $k=j$, $\ell=1$). Moreover, such an $\SS_{q_j}$ may be written as the disjoint union of three subsets of the same cardinality; since a given $\SS_{q_j}$ is associated with at most three $\TT_j$s, we may define injections $p'_j:\TT_j\to\SS_{q_j}$ with disjoint ranges, and glue them together to obtain the desired~$p'$.
Finally,  we define $p:\TT'\setminus \TT''\to \PP(\SS)$ by $p(t)=\{p'(t)\}$ (the one-element subset). Note that in the end we have $p(t)\subset \tilde \SS$
for $t\in\TT''$ 
   and $p(t)\subset \SS\setminus \tilde \SS$ for  $t\in\TT'\setminus \TT''$.

We claim that for any choice of $t_1,\dots, t_m$ the union of the finite sets $p(t_1),\dots, p(t_m)$ contains at least $m$ elements. Indeed,  if all $t_i$ belong to $\TT''$ this follows from~\eqref{eq:comb2-1}, while   if all $t_i$ belong to $\TT'\setminus\TT''$ it is a consequence of the injectivity of the function $p'$. From here follows  the general case, since   $p(t)\subset \tilde \SS$ for $t\in\TT''$ and $p(t)\subset \SS\setminus \tilde \SS$ for  $t\in\TT'\setminus \TT''$.

We may then apply Hall's marriage lemma (see, for instance,~\cite{Ry}) to obtain an injective function $\phi:\TT'\to\SS$ with the property that $\phi(t)\in p(t)$ for all $t\in\TT'$. Since $p(t)\subset\SS_{j-1}\cup \SS_j\cup\SS_{j+1}$ for $t\in\TT_j$, we have
\begin{equation}\label{eq:good bounds for phi}
\delta^2\le \frac{\tau(t)}{\sigma(\phi(t))} \le \frac{1}{\delta^2}
\end{equation}
for all $t\in\TT'$.

A similar argument, using~\eqref{eq:comb2-2}, can be used to produce, if $\SS'=\SS\setminus \SS_0$, an injective function $\psi:\SS'\to\TT$ such that for all $s\in\SS'$ we have
\begin{equation}\label{eq:good bounds for psi}
\delta^2\le \frac{\tau(\psi(s))}{\sigma(s)} \le \frac{1}{\delta^2}
\end{equation}
for all $s\in\SS'$.

Define then $\Phi:\PP(\TT)\to \PP(\TT)$ by the formula
\[
\Phi(E)= \TT\setminus \psi \Big[ \big(\SS \setminus \phi(E\cap \TT')\big)\cap \SS' \Big].
\]
$\Phi$ is then a monotone map, which has a fixed point by the classical Cantor--Bernstein argument; we will denote this fixed point by $E_0\subset \TT$. Consider the two partitions $\TT=F_1\cup F_2\cup F_3$, $\SS=G_1\cup G_2\cup G_3$, where
\[
\begin{split}
&F_1=\TT\setminus E_0, \quad F_2=E_0\cap \TT', \quad F_3= E_0\setminus \TT',\\
&G_1=\big(\SS \setminus \phi(E_0\cap \TT')\big)\cap \SS',
\quad G_2= \phi(E_0\cap \TT'), \quad G_3=\SS\setminus(G_1\cup G_2) .
\end{split}
\]
Then $\psi^{-1}$ is one-to-one from $F_1$ onto $G_1$, $\phi$ is one-to-one from $F_2$ onto $G_2$, while $F_3\subset \TT\setminus\TT'$, $ G_3\subset \SS\setminus \SS'$. Denote then $\eta_0:F_1\cup F_2\to G_1\cup G_2$ by $\eta_0(t)=\psi^{-1}(t)$ if $t\in F_1$ and $\eta_0(t)=\phi(t)$ if $t\in F_2$. From~\eqref{eq:good bounds for phi} and~\eqref{eq:good bounds for psi} it follows that
\begin{equation}\label{eq:good bounds for eta_0}
\delta^2\le \frac{\tau(t)}{\sigma(\eta_0(t))}\le \frac{1}{\delta^2} \text{ for }t\in F_1\cup F_2.
\end{equation}

This is just the desired relation for $\delta'=\delta^2$ and $t\in F_1\cup F_2$. On the other hand, for any $t\in F_3$, $s\in G_3$ we have 
\begin{equation}\label{eq:good bounds on F3G3}
\frac{\delta^N}{M}\le \frac{\tau(t)}{\sigma(s)} \le \frac{M}{\delta^N}
\end{equation}
since   $F_3\subset \TT\setminus\TT'$ and $ G_3\subset \SS\setminus \SS'$ imply $\tau(t), \sigma(s)\in[\delta^N,M]$.

Take then $\delta'=\min\{\delta^2, \delta^N/M\}$. If $\#F_3=\# G_3$, we may extend $\eta_0$ to a bijection $\eta:\TT\to \SS$, and the desired inequalities~\eqref{eq:delta' inequalities} are true    by~\eqref{eq:good bounds for eta_0} and~\eqref{eq:good bounds on F3G3}.

If $F_3$ and $G_3$ do not have the same cardinality, say  $\#F_3<\#G_3$, we take $\II$ such that   $\#(F_3\cup\II)=\#G_3$. We may then extend $\eta_0$ to a bijection $\eta: \TT\cup \II\to \SS$ that satisfies $\eta(F_3\cup\II)=G_3$, and again~\eqref{eq:delta' inequalities} is satisfied. A similar argument works if $\#F_3>\#G_3$, leading to case (iii) from the statement of the theorem.

We have thus proved (1). For (2), we notice that if~\eqref{eq:combinatorial lemma1} and~\eqref{eq:combinatorial lemma2} are true for all $k\in\bbZ$, then we may define $p$ as above on the whole of $\TT$ instead of only on $\TT'$; consequently, we obtain an injective function $\phi:\TT\to\SS$. Similarly, we get $\psi:\SS\to\TT$; then the usual Cantor--Bernstein argument is used, leading to $F_3=G_3=\emptyset$ and $\eta_0=\eta$. So no extension of $\eta$ is needed, and~\eqref{eq:delta' inequalities} is satisfied.
\end{proof}

\section{Strong equivalence and equivalence after extension}\label{se:strong equivalence}

The discussion in this section of strong equivalence and equivalence after extension is mostly  based on~\cite{FW}.

\begin{lemma}\label{le:simple reduction to selfadjoints}
If $T,S$ are two linear operators, then the following are equivalent:
\begin{itemize}
\item
[(i)] 
$T\sim S$.
\item
[(ii)] $\dim\ker T=\dim\ker S$ and there exists a unitary operator $U$ such that $U(\ran T)=\ran S$.

\item
[(iii)] Conditions~\eqref{eq:index} are satisfied, and the restriction of $|T|$ to $ \supp T$ is similar to the restriction of $  |S|$ to $ | \supp S$.
\item
[(iv)] Conditions~\eqref{eq:index} are satisfied, and  there exists a unitary operator $V$ such that $V(\ran|T|)=\ran|S|$.

\end{itemize}
\end{lemma}

\begin{proof}
The equivalence of (i) and (ii) is~\cite[Theorem 3.4]{FW}. The equivalences (i) $\Leftrightarrow$(iii) and (ii)$\Leftrightarrow$(iii) are simple consequences of the polar decomposition of $T$ and~$S$.
\end{proof}

The previous lemma essentially reduces the problem of strong equivalence to the case of positive operators: we have to know when their ranges can be mapped one onto the other by means of a unitary operator. This is expressed in terms of  their respective spectral measures in the next result, for which we need first a definition.

\begin{definition}\label{de:condition S}
Two positive operators $A\in\LL(X)$, $B\in\LL(Y)$ \emph{satisfy condition $(\mathfrak{S})$} if
there exists $0<\delta<1$ such that for any $0<\alpha\le \beta<\infty$ we have
\begin{align}
\dim E_A([\alpha,\beta))X&\le \dim E_B(\big[\delta\alpha, \frac{1}{\delta}\beta\big))Y\label{eq:se1},\\
\dim E_B([\alpha,\beta))Y&\le \dim E_A(\big[\delta\alpha, \frac{1}{\delta}\beta\big))X\label{eq:se2}.
\end{align}
\end{definition}

The next lemma follows from~\cite[Lemma 3.2]{FW}. 

\begin{lemma}\label{le:equivalence of selfadjoints}
Suppose $A,B$ are positive. If there exists a unitary operator $V$ such that $V(\ran A)=\ran B$ (in particular, if $A\sim B$), then
$A,B$ satisfy condition $(\mathfrak{S})$. 
 \end{lemma}

To deal with equivalence after extension, we 
  formulate for positive operators another condition, less stringent that $(\mathfrak{S})$.

\begin{definition}\label{de:condition tilde S}
Two positive operators $A\in\LL(X)$, $B\in\LL(Y)$ \emph{satisfy condition $(\widetilde{\mathfrak{S}})$} if
there exists $0<\delta<1$ and $a>0$ such that for any $0<\alpha<\beta\le a$ we have
\begin{align}
\dim E_A([\alpha,\beta))&\le \dim E_B(\big[\delta\alpha, \frac{1}{\delta}\beta\big))\label{eq:see1}\\
\dim E_B([\alpha,\beta))&\le \dim E_A(\big[\delta\alpha, \frac{1}{\delta}\beta\big))\label{eq:see2}
\end{align}

\end{definition}

The difference between $({\mathfrak{S}})$ and $(\widetilde{\mathfrak{S}})$ is that the latter involves only  spectral projections   supported on $(0,a/\delta)$. As a consequence, we have the following result:

\begin{lemma}\label{le:s and tilde s}
Suppose $A\in\LL(X)$, $B\in\LL(Y)$ are two positive operators. The following are equivalent:
\begin{itemize}
\item
[(i)] $A,B$ satisfy condition $(\widetilde{\mathfrak{S}})$.

\item[(ii)]
$A\oplus I_Y$ and $B\oplus I_X$ satisfy condition $(\mathfrak{S})$.
\end{itemize}

\end{lemma}

\begin{proof}
Suppose $A,B$ satisfy condition $(\widetilde{\mathfrak{S}})$ with some $\delta, a$. We claim that $(\mathfrak{S})$ is true for $A\oplus I_Y$ and $B\oplus I_X$, replacing $\delta$ by $\delta'=\min \{\delta, \frac{a}{\|A\|}, \frac{a}{\|B\|}\}$. Indeed, by $(\widetilde{\mathfrak{S}})$, both conditions $(\mathfrak{S})$ are satisfied if $\beta<a$. But, if $\beta\ge a$, then the spectral projection in the right hand side in~\eqref{eq:se1} and~\eqref{eq:se2} is the whole direct sum, its dimension is  $\dim X+\dim Y$, and thus the inequalities are satisfied.

Conversely, suppose  $A\oplus I_Y$ and $B\oplus I_X$ satisfy condition $(\mathfrak{S})$. We keep the same value of $\delta$ and choose $a=\delta$. Then conditions~\eqref{eq:see1} and~\eqref{eq:see2} involve only spectral projections of sets contained in $(0,1)$, and these  are not influenced by a direct summand that is an identity operator. So  they are also satisfied by  $A, B$.
\end{proof}

\section{Equivalence after one-sided extension}\label{se:eoe}

This section contains the main result of the paper, namely the implication (EE)\Implies(EOE).  The basic result is the next lemma, which uses the construction of Lemma~\ref{le:combinatorics}.

\begin{lemma}\label{le:main argument}
Suppose $A\in\LL(X)$, $B\in\LL(Y)$ are positive operators with $\dim\ker A=\dim\ker B$. 
\begin{enumerate}
\item
If $A,B$ satisfy condition $(\widetilde{\mathfrak{S}})$, then $A$ and $B$ are equivalent after one-sided extension. 
\item
If $A,B$ satisfy condition $(\mathfrak{S})$, then $A\sim B$.
\end{enumerate}
\end{lemma}

\begin{proof} Let us first note that $\dim\ker A=\dim\ker B$ says that $A,B$ fulfil the kernel condition~\eqref{eq:index}.

(1) Suppose~$(\widetilde{\mathfrak{S}})$ is satisfied by $A$ and $B$ for some $\delta<1$ and $a>0$; since decreasing $a$ preserves the inequalities, we may assume  $a=\delta^N$ for some $N\ge 1$. In particular,  taking, for $k\ge N$, $\alpha=\delta^{k+\ell}, \beta=\delta^{k}$, we obtain
\begin{align}
\dim E_{A}([\delta^{k+\ell},\delta^{k}))X&\le \dim E_{B}(\big[\delta^{k+\ell+1}, \delta^{k-1}))Y\label{eq:main1},\\
\dim E_{B}([\delta^{k+\ell},\delta^{k}))Y&\le \dim E_{A}(\big[\delta^{k+\ell+1}, \delta^{k-1}))X\label{eq:main2}.
\end{align}

Let $\|A\|, \|B\|< M$, and define  $\TT_j$ to be an orthonormal basis of  $E_{A}([\delta^{j+1}, \delta^j))X$ for $j\in \bbZ$, $\TT=\bigcup_{j\in\bbZ}^\infty \TT_j$. (Note that $\TT_j=\emptyset$ for $\delta^j\ge M$.) Obviously $\TT$ is an orthonormal basis for $\supp A$.
Define, for $t\in\TT_j$, $\tau(t)=\delta^{j+1}$; we have then, by  construction, $\TT_j=\tau^{-1}([\delta^{j+1}, \delta^j)$.

Similarly, we define $\SS_j$ to be an orthonormal basis of  $E_{B}([\delta^{j+1}, \delta^j))$, $\SS=\bigcup_{j\in\bbZ} \SS_j$ (so $\SS$ is an orthonormal basis for $\supp B$), and $\sigma(s)=\delta^{j+1}$ for $s\in \SS_j$.

We may then apply to the above objects Lemma~\ref{le:combinatorics}, and we obtain as conclusion one of the three cases (i), (ii), (iii) from its statement; we may also assume that $\delta'=\delta^q$ for some positive integer~$q$.

Suppose (i) is true, and we have thus a bijection $\eta:\TT\to \SS$, such that
\[
\delta^q\le \frac{\tau(t)}{\sigma(\eta(t))}\le \delta^{-q}.
\]
The above inequalities say that if $t$ is an element of an orthonormal basis of $E_A([\delta^k, \delta^{k+1}))X$, then $\eta(t)$ is an element of an orthonormal basis of $E_B([\delta^j, \delta^{j+1}))Y$ with $|k-j|\le q$.
By applying the criterium of  Lemma~\ref{le:range of selfadjoint}, it follows that the unitary operator $U:\supp A\to \supp B$ defined by $U(t)=\eta(t)$ maps $\ran A$ precisely onto $\ran B$.  As the kernel condition is satisfied by hypothesis, Lemma~\ref{le:simple reduction to selfadjoints} implies  that $A\sim B$.

In case (ii) the same argument works after we  add first to the domain of $A$ a direct summand $H$ of dimension equal to $\#\II$, and then  consider $A\oplus I_H$ instead of $A$. We obtain then $A\oplus I_H\sim B$. Case (iii) is similar, and thus (1) is proved.

To prove (2), we have to note that we are in the case covered by Lemma~\ref{le:combinatorics} (2), and therefore (i) is true, which by the above reasoning leads to $A\sim B$.
\end{proof}

\begin{remark}\label{re:dimension of the extension}
It follows from the proof  that the dimension of the extension space in case (1) is at most the maximum of the cardinals of $\TT\setminus \TT'$ and $\SS\setminus \SS'$, that is, the dimensions of the spectral projections of the two operators corresponding to the set $[\delta^N, \infty)$.

\end{remark}

Using  Lemmas~\ref{le:simple reduction to selfadjoints} and~\ref{le:equivalence of selfadjoints}, we obtain  the following characterisation of equivalence, that coincides essentially with the one in~\cite{FW}.

\begin{theorem}\label{th:equivalence in general}
Suppose $T,S$ are two linear operators. Then the following assertions are equivalent:
\begin{itemize}
\item[(i)]
$T\sim S$.
\item
[(ii)] Relations~\eqref{eq:index} are true and $|T|, |S|$ satisfy condition $(\mathfrak{S})$.
\item
[(iii)]
Relations~\eqref{eq:index} are true  and the restrictions of $|T|, |S|$ on their corresponding supports are equivalent.
\end{itemize}
\end{theorem}

We arrive now at the main result of this section.

\begin{theorem}\label{th:main}
If $T,S$ are linear operators, then the following are equivalent:
\begin{itemize}
\item
[(i)] $T$ and $S$ are equivalent after one-sided extension.
\item
[(ii)] $T$ and $S$ are Schur coupled.
\item
[(iii)] $T\sime S$.
 \item[(iv)] Relations~\eqref{eq:index} are true and $|T|, |S|$ satisfy  condition $(\widetilde{\mathfrak{S}})$.
\end{itemize}
\end{theorem}

\begin{proof}
According to the  remarks in Section~\ref{se:prelim}, we  know that (i)\Implies(ii)\Implies(iii). We have already noticed that if $T\sime S$ then~\eqref{eq:index} are true.  Also, $|T|\sime |S|$ implies, by Lemma~\ref{le:equivalence of selfadjoints}, that $|T|\oplus I_Y$ and $|S|\oplus I_X$  satisfy $(\mathfrak{S})$, whence by Lemma~\ref{le:s and tilde s}  $|T|, |S|$ satisfy   $(\widetilde{\mathfrak{S}})$. Thus (iii)\Implies(iv).

Finally, if $|T|, |S|$ satisfy  condition $(\widetilde{\mathfrak{S}})$, it follows by Lemma~\ref{le:main argument} that they are equivalent after one-sided extension. Together with relations~\eqref{eq:index}, this implies that $T,S$ are equivalent after one-sided extension.
\end{proof}

We may recapture easily a result proved in~\cite[Theorem 1.3]{HR}.

\begin{corollary}\label{co:closed range}
Suppose $T,S$ have closed range. Then the validity of relations~\eqref{eq:index} is equivalent to any of the conditions in Theorem~\ref{th:main}.
\end{corollary}

\begin{proof}
To see that (iv) in Theorem~\ref{th:main} is true, note that the assumption implies that $E_{|T|}((0,\epsilon))$ and $E_{|S|}((0,\epsilon))$ are 0 for some $\epsilon>0$, and so $|T|, |S|$ obviously satisfy  condition $(\widetilde{\mathfrak{S}})$.
\end{proof}

\section{Characterization of Schur coupling for some classes of operators}\label{se:ch schur}

The equivalence relations considered may be given a more concrete characterisation in the case of compact operators, in terms of their singular values.
It is convenient to introduce the following definition for sequences of positive numbers.

\begin{definition}\label{de:comparable}
Suppose $(t_n)_{n\ge0}$, $(s_n)_{n\ge0}$ are two sequences of positive numbers. We will say that they are \emph{comparable after a shift} if there 
exists $0<\delta<1$ and  $m\in\bbN $ such that either 
\[
\delta\le\frac{s_n}{t_{n+m}}\le\frac{1}{\delta} \text{ for all } n\ge0,
\]
or
\[
\delta\le\frac{t_n}{s_{n+m}}\le\frac{1}{\delta} \text{ for all } n\ge0.
\]

If the above relations are true for $m=0$, we will simply say that the two sequences are \emph{comparable}.
\end{definition}

It is clear that comparability after a shift is not affected by adding or deleting a finite number of values to any of the sequences.

For the case of compact operators a precise characterization of equivalence is obtained in the next theorem (probably well-known).

\begin{theorem}\label{th:sing}
Suppose $T,S$ are compact, the singular values of $T$ are $t_i\searrow 0$ and the singular values of $S$ are $s_i\searrow 0$. Then the following are equivalent:

{\rm (i)} $T\sim S$.

{\rm (ii)} Relations~\eqref{eq:index} are satisfied and the sequences $(t_i)$, $(s_i)$ are comparable.
\end{theorem}

\begin{proof}
The singular values of an operator are the eigenvalues of its modulus. Therefore, by Theorem~\ref{th:equivalence in general}, we have to prove that for the positive compact operators $|T|, |S|$ condition $(\mathfrak{S})$ is equivalent to comparability of their respective eigenvalues; that is, the existence of $0<\delta<1$ such that, for all $i\in\bbN$,
\begin{equation}\label{eq:(ii)(2)}
\delta\le\frac{t_i}{s_i}\le 1/\delta.
\end{equation}

Suppose~\eqref{eq:(ii)(2)} is satisfied and take $0<\alpha\le\beta<\infty$. The interval $[\alpha,\beta)$ contains a finite number of eigenvalues of $|T|$, say (in decreasing order and taking into account multiplicities) $t_p\ge\dots\ge t_{p+q}$.  By~\eqref{eq:(ii)(2)} it follows that $s_i\in [\delta\alpha, \beta/\delta)$ for all $i=p,\dots,p+q$, whence~\eqref{eq:se1} is satisfied for $A=|T|$, $B=|S|$. A similar argument yields~\eqref{eq:se2}. 

Conversely, suppose $(\mathfrak{S})$ is satisfied for $A=|T|$, $B=|S|$. To prove~\eqref{eq:(ii)(2)}, fix $i\in\bbN$, and suppose that $j_i=\max\{j\in\bbN: t_j=t_i\}$. If $\alpha=t_{j_i}$  $\beta>\|T\|$, applying~\eqref{eq:se1}, it follows that $\dim E_{|S|}([\delta t_{j_i}, \frac{\beta}{\delta}))\ge j_i$. Therefore  $\delta t_{j_i}\le s_{j_i} $, whence $\delta t_i=\delta t_{j_i}\le s_{j_i}\le s_i$, which yields one of the inequalities~\eqref{eq:(ii)(2)}. The other is proved similarly.
\end{proof}

Theorem~\ref{th:sing} has consequences for Schur coupling. We start with the case of compact operators. 
\begin{theorem}\label{th:schur coupling for compacts}
Suppose $T$ and $S$ are compact. Then any of the conditions in Theorem~\ref{th:main} is equivalent to 

\begin{itemize}
\item
[(v)] Equalities~\eqref{eq:index} are satisfied and the singular values of $T$ and $S$ are comparable after a shift.

\end{itemize}
\end{theorem}

\begin{proof}
In the case of compact positive operators the spectral projections corresponding to any set of the form $[\delta,\infty)$ are finite dimensional. By Remark~\ref{re:dimension of the extension}, it follows that the extension space is finite dimensional, and thus, for instance, $T\oplus I_H\sim S$, with $m:=\dim H<\infty$. Applying Theorem~\ref{th:sing} to these operators finishes the proof.
\end{proof}

This also allows us to recapture the following result.

\begin{corollary}[\cite{H}]\label{co:ideals}
If $T\sime S$ and $T$ belongs to some ideal of compact operators, then $S$ belongs to the same ideal.
\end{corollary}

A slight modification of the argument yields a characterization of Schur coupling for the  case in which one operator is compact and the other is not; this is the content of the next statement.
 
\begin{theorem}\label{th:main description2}

Suppose  $T\in\LL(X)$ is compact, while $S\in\LL(Y)$ is not. Then $T\sime S$  if and only if they fulfill the following conditions: 
\begin{enumerate}
\item
Equalities~\eqref{eq:index} are satisfied.

\item
There exists $\epsilon>0$ such that  $|S|E_{|S|}([0,\epsilon])$ is compact.

\item
If  $(s_n)_{n\ge0}$ are the eigenvalues of $|S|E_{|S|}([0,\epsilon])$ and $(t_n)_{n\ge0}$ are the singular values of $T$, then they are comparable after a shift.
\end{enumerate}

\end{theorem}

\begin{proof}  
If $T,S$ are Schur coupled, and thus equivalent after one-sided extension, we must have $T\oplus I_H\sim S$ (the other possibility would lead to the equivalence of a compact operator with a non compact one). By Theorem~\ref{th:equivalence in general} (1) is satisfied; moreover, applying condition $(\mathfrak S)$ for any $\alpha>0$ and $\beta=\delta/2$, it follows that 
\[
\dim E_{|S|}([\alpha, \delta/2))Y\le \dim E_{|T|\oplus I_H}([\delta\alpha, 1/2))X= \dim E_{|T|}([\delta\alpha, 1/2))<\infty.
\]
Therefore, if we define $\epsilon=\delta/2$, then $|S|E_{|S|}([0,\epsilon))$ is compact.

Consider then the operators $T'=|T|E_{|T|}([0,\epsilon))$, $S'=|S|E_{|S|}([0,\epsilon))$. Then  $T', S'$ satisfy the assumptions of  Theorem~\ref{th:main} (iv), whence $T'\sime S'$. Since they are both compact, Theorem~\ref{th:schur coupling for compacts} implies that their eigenvalues are comparable after a shift. But the singular values of $T$ are obtained by adding a finite number of values to the eigenvalues of $T'$, and so it follows that condition (3) from the statement is satisfied.

Conversely, suppose (1)--(3) are satisfied. Write $|S|=S'\oplus S''$, with $S'=|S|$ restricted to $ E_{|S|}([0,\epsilon])$. Then Theorem~\ref{th:schur coupling for compacts} says that $S'\sime T$, while $S''$ is an invertible operator  and so $S''\sime I_H$ for some Hilbert space $H$. It follows that $|S|\sime T\oplus I_Y$. Condition (1) implies then that $S\sime T\oplus I_H$, and therefore $T,S$ are Schur coupled. 
\end{proof}

We complete the section with a result pertaining to the case of both operators noncompact.

\begin{theorem}\label{th:separable noncompact}
Suppose $T\in\LL(X)$, $S\in\LL(Y)$, both noncompact, and  $\dim X=\dim Y$. If $T\sime S$, then $T\sim S$.
\end{theorem}

\begin{proof}
If $T$ is not compact, then there exists $\delta>0$ such that the spectral projection $E_\delta=E([\delta,\|T\|])$ of $|T|$ is infinite dimensional. If we denote $T_0=T|E_\delta X$ with values in $T(E_\delta X )$, then $T_0\sim T_0\oplus I_Y$ by Lemma~\ref{le:immediate remarks equivalence} since both are invertible on spaces of the same dimension. It follows  that $T\sim T\oplus I_Y$. Similarly we obtain $S\sim S\oplus I_X$.

Now, if $T\sime S$, then $T\oplus I_Y\sim S\oplus I_X$, and 
\[ 
T\sim T\oplus I_Y\sim S\oplus I_X\sim S.\qedhere
\]
\end{proof}

\section{Final remarks}

If $T,S$ are compact, by Theorem~\ref{th:schur coupling for compacts} they are equivalent after a one-sided extension of finite rank. The rank might or might not be uniquely determined, as shown by the following examples.

\begin{enumerate}
\item
If $T=S$ is the diagonal operator with eigenvalues $(\frac{1}{n})$, then $T\sim S\oplus I_H$ for any finite dimensional space~$H$.

\item
If $T=S$  is the diagonal operator with eigenvalues $(\frac{1}{n!})$, then $T\sim S\oplus I_H$ implies $H=\{0\}$.
\end{enumerate}

For finite dimensional spaces, the characterization of Schur coupling appears already in~\cite[Theorem 1]{BT2}: if $T\in\LL(X)$ and $S\in\LL(Y)$, then $T\sime S$  if and only if $\dim X-\rank T=\dim Y-\rank S$. Note that this is also a particular case of Corollary~\ref{co:closed range}.

\section*{Acknowledgements}
The  author is  supported by a grant of the Romanian National Authority for Scientific
Research, CNCS Ð UEFISCDI, project number PN-II-ID-PCE-2011-3-0119.

\end{document}